\documentclass[letterpaper, 10 pt, twoside]{ieeeconf}

\IEEEoverridecommandlockouts

\def\BibTeX{{\rm B\kern-.05em{\sc i\kern-.025em b}\kern-.08em
    T\kern-.1667em\lower.7ex\hbox{E}\kern-.125emX}
}%

\newcommand{\BIBTEXDIR}{OL_bibtex}

 \usepackage{bbm}

\newcommand{\I}{\mathcal{I}}

\newcommand{\T}{\mathcal{T}}

\usepackage{booktabs}

\newcommand{\agent}{s}

\usepackage{subcaption}

\newcounter{inlineenum}
\renewcommand{\theinlineenum}{\alph{inlineenum}}

\usepackage{ntheorem}
\newtheorem*{task*}{Problem}
\newtheorem*{assumptions*}{Assumptions}
\newtheorem*{notation*}{Notation}

\usepackage{amssymb,amsfonts}
\usepackage{algorithmic}
\usepackage{mathtools}
\usepackage{graphicx}
\usepackage{textcomp}
\usepackage{hyperref}
\usepackage{lmodern}
\usepackage[utf8]{inputenc}
\usepackage{graphicx}
\usepackage[export]{adjustbox}
\usepackage{cases}
\usepackage{empheq}
\usepackage{fancyhdr}
\usepackage{rotating}
\usepackage{bigdelim}
\usepackage{bm}
\usepackage{bbm}
\usepackage{tikz}
\usepackage{etoolbox}
\usepackage[short]{optidef}
\usepackage{multirow}
\usepackage{bookmark}
\usepackage{comment}

\renewcommand{\L}{\mathcal{L}}

\newcommand{\secref}[1]{Sec.~\ref{#1}}
\newcommand{\lemref}[1]{Lemma~\ref{#1}}

\newcommand{\theoref}[1]{Thm.~\ref{#1}}
\newcommand{\propref}[1]{Prop.~\ref{#1}}

\renewcommand{\eqref}[1]{(\ref{#1})}
\newcommand{\figref}[1]{Fig.~\ref{#1}}

\newcommand{\ocpref}[1]{\ac{ocp}~\eqref{#1}}
\newcommand{\corref}[1]{Corollary~\ref{#1}}

\newcommand{\myqed}{\hfill~\qed}

\newcommand{\R}{\mathbb{R}}

\newcommand{\ift}{\mathrm{~if~}}

\newcommand{\othert}{\mathrm{~otherwise}}

\newcommand{\andt}{\mathrm{~and~}}

\usepackage{upgreek}

\usepackage[ruled,vlined,linesnumbered]{algorithm2e}
\newcommand{\dt}{\textup{d}t}

\renewcommand{\theenumi}{(\roman{enumi})}%

\newtheorem{theorem}{Theorem}

\newtheorem{proposition}{Proposition}
\newtheorem{lemma}{Lemma}
\newtheorem{corollary}{Corollary}

\newtheorem{remark}{Remark}

\usepackage[acronym,hyperfirst=false]{glossaries}

\newacronym[plural=LCQPs,firstplural=Linear Complementarity Quadratic Programs (LCQP)]{lcqp}{LCQP}{Linear Complementarity Quadratic Program}
\newacronym[plural=MPCCs,firstplural=Mathematical Programs with Complemenetarity Constraints (MPCC)]{mpcc}{MPCC}{Mathematical Program with Complemenetarity Constraints}
\newacronym[plural=QPs,firstplural=Quadratic Programs (QP)]{qp}{QP}{Quadratic Program}
\newacronym[plural=QCQPs,firstplural=Quadratically Constrained Quadratic Programs (QPQP)]{qcqp}{QCQP}{Quadratically Constrained Quadratic Program}
\newacronym[plural=NLPs,firstplural=Nonlinear Programs (NLP)]{nlp}{NLP}{Nonlinear Program}
\newacronym[plural=OCPs,firstplural=Optimal Control Problems (OCP)]{ocp}{OCP}{Optimal Control Problem}
\newacronym[plural=MIQPs,firstplural=Mixed Integer Quadratic Programs (MIQP)]{miqp}{MIQP}{Mixed Integer Quadratic Program}
\newacronym[plural=MINLPs,firstplural=Mixed Integer Non-Linear Programs (MINLP)]{minlp}{MINLP}{Mixed Integer Non-Linear Program}
\newacronym[plural=QPCCs,firstplural=Quadratic Programs with Complementarity Constraints (QPCC)]{qpcc}{QPCC}{Quadratic Program with Complementarity Constraints}
\newacronym[plural=QPLCCs,firstplural=Quadratic Programs with Linear Compementarity Constraints (QPLCC)]{qplcc}{QPLCC}{Quadratic Program with Linear Compementarity Constraints}
\newacronym[plural=FDIs,firstplural=Filippov Differential Inclusions (FDI)]{fdi}{FDI}{Filippov Differential Inclusion}
\newacronym[plural=ODEs,firstplural=Ordinary Differential Equations (ODE)]{ode}{ODE}{Ordinary Differential Equation}
\newacronym[plural=TSPs,firstplural=Traveling Salesperson Problems (TSP)]{tsp}{TSP}{Traveling Salesperson Problem}

\newacronym{sqp}{SQP}{Sequential Quadratic Programming}
\newacronym{scp}{SCP}{Sequential Convex Programming}
\newacronym{licq}{LICQ}{Linear Independence Constraint Qualification}
\newacronym{mfcq}{MFCQ}{Mangasarian-Fromovitz Constraint Qualification}
\newacronym{kkt}{KKT}{Karush-Kuhn-Tucker}
\newacronym{pm}{PM}{Persistent Monitoring}
\newacronym{pmp}{PMP}{Pontryagin Minimum Principle}
\newacronym{lls}{LLS}{Linear Least Squares}
\newacronym{ipa}{IPA}{Infinitesimal Perturbation Analysis}
\newacronym{ad}{AD}{Automatic Differentiation}
\newacronym{mpc}{MPC}{Model Predictive Control}

\newcommand{\ac}[1]{\gls*{#1}}
\newcommand{\acp}[1]{\glspl*{#1}}

\title{A Bilevel Optimization Scheme for Persistent Monitoring}

\author{\IEEEauthorblockN{}
\IEEEauthorblockA{\textit{Systems Engineering} \\
\textit{Boston University}\\
Massachusetts, USA
}
}

\author{Jonas Hall$^{1}$, Logan E. Beaver$^{1}$, Christos G. Cassandras$^{1,2}$, Sean B. Andersson$^{1,3}$ 
\thanks{This work was supported in part by NSF under grants ECCS-
1931600, DMS-1664644, CNS-1645681, CNS-2149511, by AFOSR under FA9550-19-1-0158, by ARPA-E under DE-AR0001282, and by the MathWorks.}%
\thanks{$^{1}$Division of Systems Engineering, Boston University, USA}%
\thanks{$^{2}$Department of Electrical and Computer Engineering}%
\thanks{$^{3}$Division of Mechanical Engineering, Boston University, USA}%
\thanks{{\tt\small \{hallj, \dots \}@bu.edu}}
}%

\begin{document}

\clearpage\maketitle

\begin{abstract}
    In this paper we study an infinite-horizon persistent monitoring problem in a two-dimensional mission space containing a finite number of statically placed targets, at each of which we assume a constant rate of uncertainty accumulation. Equipped with a sensor of finite range, the agent is capable of reducing the uncertainty of nearby targets. We derive a steady-state minimum time periodic trajectory over which each of the target uncertainties is driven down to zero during each visit. A hierarchical decomposition leads to purely local optimal control problems, coupled via boundary conditions. We optimize both the local trajectory segments as well as the boundary conditions in an on-line bilevel optimization scheme. 
\end{abstract}

\section{Introduction}

Monitoring a dynamically changing environment in an efficient and cost-feasible manner has long since attracted attention due to its broad applicability to areas such as ocean monitoring~\cite{smith2011persistentOcean, alam2018data}, forest monitoring~\cite{naderi2022sharing}, wildfire surveillance~\cite{casbeer2006cooperative}, data harvesting~\cite{zhu2022control, lee2006efficient}, or particle tracking~\cite{pinto2021tracking}. A  common approach is to place static sensors in order to maximize the monitored area or to maximize event detection probability, which in the literature is known as the coverage control problem~\cite{cortes2004coverage}. However, employing a large number of static sensors can be expensive and inflexible. Hardware and software advances have enabled the replacement of static sensors by equipping the sensors to autonomous agents. The coverage control problem was thus extended to the \ac{pm} problem~\cite{cassandras2011optimal}.

Over the last decade this problem has accumulated a rich set of formulations and variations. For some formulations the dynamic environment consists of a connected and typically compact subset of $\R^n$. In this setting the agents are often tasked to detect rogue elements appearing at unknown locations~\cite{boldrer2022time}, or to minimize the cumulative average value of a dynamically changing scalar field~\cite{lin2014optimal}. Other formulations, as is the case in this paper, focus on a finite set of targets within the environment. Typical tasks then consist of detecting stochastic events at known locations~\cite{yu2015persistent}, or minimizing the maximum revisit time along a periodic trajectory~\cite{hari2020optimal}. Usually the targets are spatially static, however, some formulations consider mobile targets as well~\cite{anderson2014stochastic,hall2022optimal,wang2023spatio}.

A common subproblem of \ac{pm} tasks consists of determining a periodic visiting sequence of all the targets, which in and of itself is NP-hard since it is more general than the \ac{tsp}, due to the dynamic nature of the problem. Even if a good visiting sequence is determined, computing optimal agent trajectories (with respect to a given metric such as minimum time or minimum energy) remains challenging. In order to monitor a given target we require the agent to be close to it. However, the more time the agent spends monitoring one target, the more cost is accumulated at all other targets. On the other hand, if the agent moves too quickly past a target, then the local cost, and thus also the global cost is insufficiently reduced. A challenge in designing trajectories is to manage this trade-off. Due to the difficulties of solving \ac{pm} problems, they are often decomposed and many contributions focus on specific subproblems. One such decomposition is the path-velocity decomposition~\cite{kant1986toward}. 
However, this decomposition is always suboptimal unless the agent's local sensing capability is independent of the target-agent distance. Examples for velocity controllers along a given path can be found in~\cite{smith2011persistent,song2014optimal}. The vast majority of methods for trajectory optimization work off-line~\cite{lin2014optimal,ostertag2022trajectory}, however the authors of~\cite{notomista2021online} introduced an on-line trajectory optimization approach. Inherently different to the approach of decomposition is that of abstraction. Such methods formulate the mission space using a graph topology, where each target is described as a node and edges between two targets reflect the travel cost between those targets~\cite{alamdari2014persistent,welikala2021event,hari2020optimal}. Such methods aim at solving the target visiting sequence, instead of directly controlling the agents.

In this paper we consider a \ac{pm} formulation with a single agent and $M$ targets, each of which is associated with an internal state that models uncertainty. The goal is to minimize the infinite-horizon average uncertainty. We introduce a method that optimizes the agent's trajectory on-line. Similar to~\cite{ostertag2022trajectory}, we decompose the problem into purely local \acp{ocp}, the solutions of which provide decoupled trajectory segments. We then solve these \acp{ocp} using a direct multiple shooting approach~\cite{rao2009survey}. While modern solvers are able to treat optimal control problems with hybrid dynamics directly~\cite{nurkanovic2022nosnoc,hall2021sequential}, we utilize this decomposition for the simple reason that the dimension of the local problems become independent of the number of targets. The contributions of this paper are 
\begin{enumerate}
    \item \theoref{theorem:exact:reformulation}, which shows the existence of an exact relaxation for the local \ocpref{ocp:local} with hybrid dynamics, and
    \item a bilevel optimization scheme that optimizes the agent's trajectory on-line. 
\end{enumerate}

The remainder of this paper is organized as follows. \secref{sec:problem:formulation} introduces the considered \ac{pm} problem. In \secref{sec:hierarchical:decomposition} we introduce the decomposition into two layers: a sequence planner on the higher level; and a low-level layer generating the individual trajectory segments. In \secref{sec:solution:local:ocp} we analyze the low-level problems in detail, the main result being \theoref{theorem:exact:reformulation}, which shows the existence of an exact relaxation. We then utilize a gradient descent method in \secref{sec:solving:coordinator} in order to optimize the boundary conditions that are imposed on the lower levels. \secref{sec:numerical:results} discusses results in a comparison to a greedy solution.

\section{Problem Formulation}\label{sec:problem:formulation}

We are interested in a \ac{pm} problem with a single agent and $M$ targets indexed by $\T = \{1,2,\dots,M\}$. We consider first-order agent dynamics $\dot{s}(t) = u(t)$ with bounded control input $\| u(t) \| \leq 1$, where $\| \cdot \|$ denotes the Euclidean norm. The fixed positions of the targets are denoted by $x_1, x_2, \dots, x_M \in \R^2$. We assume that each target is associated with an internal state that models a measure of uncertainty $R_i$, which evolves according to the dynamics $\dot{R}_i = f_{R_i}$ given by
\begin{equation*}
    f_{R_i}(R_i, s) = \begin{cases}
        0, &~R_i = 0, A_i - B_i p_i(s) < 0, \\
        A_i - B_i p_i(s) &\othert, 
    \end{cases}
\end{equation*}
where $B_i > A_i$, and 
\begin{equation*}
    p_i(s(t)) = \max \left\{0, 1 - \frac{\| s(t) - x_i\|^2}{r_i^2}\right\}
\end{equation*}
is the monitoring model with sensing range $r_i$. It is fairly straightforward to consider more complicated sensing functions so long as they remain monotonic in the agent-target distance.

We are interested in minimizing the average uncertainty over an infinite-horizon. Solving this problem is very challenging. Previous results~\cite{cassandras2012optimal,zhou2018optimal,hall2022optimal} indicate that optimal trajectories typically drive the uncertainty of each target to zero before visiting another target. Under the assumption that this holds for each target visit, minimizing the average uncertainty is achieved by minimizing the period of a steady-state trajectory. Motivated by this behavior we formulate the following persistent monitoring problem. 

\begin{task*}
    Find a steady-state minimum time periodic trajectory over which each target uncertainty is drained, i.e., driven down to zero, during each visit.
\end{task*}

\begin{assumptions*}
Throughout this paper we assume that: 1) the initial uncertainties $R_i(0)$ are known for all targets $i \in \T$; 2) the sensing areas around the targets do not intersect; and 3) there exists a steady-state solution.
\end{assumptions*}

These assumptions are typical in the given \ac{pm} setting~\cite{smith2011persistent,zhou2018optimal,hall2022optimal}. Note that assumption 2) is fundamental for the decomposition, whereas assumption 3) is fundamental for convergence. We remark that the existence of steady-state trajectories strongly depends on the topology of the mission space as well as the parameters $A, B$, and $r$. While the existence of steady-state solutions can be proven when the uncertainty model is replaced by a Kalman filter model~\cite{pinto2022multi, ostertag2022trajectory}, this remains an open problem in the given setting.

\section{Hierarchical Decomposition}\label{sec:hierarchical:decomposition}
With the problem set up, it is a natural task to identify characterizing properties of an optimal periodic trajectory. It is immediately evident that the agent is required to visit each of the targets in order to drive their uncertainties down to zero. This understanding directly induces a two-level hierarchy: a higher level with the objective of finding a target visiting sequence; and a lower level of steering the agent so as to 1) satisfy the target visiting sequence and 2) \textit{drain} each of the target uncertainties, i.e., drive them down to zero. We will formulate the low-level problems as \acp{ocp} and we refer to those problems as the \textit{local \acp{ocp}}, since solving them only requires knowledge of local information of the visited target. In order to connect the two levels, we introduce a coordinator,\footnote{The coordinator is not to be confused with coordinators in multi-agent systems, which coordinate information between agents. Here it coordinates information between trajectory segments.} which takes a visiting sequence and then coordinates the local trajectories by providing the boundary conditions of the local \acp{ocp}. Additionally, it is the coordinator's task to optimize those boundary conditions on-line. 

\begin{figure}
    \centering
    \begin{subfigure}{0.33\linewidth}
        \centering
        \includegraphics[width=0.9\linewidth]{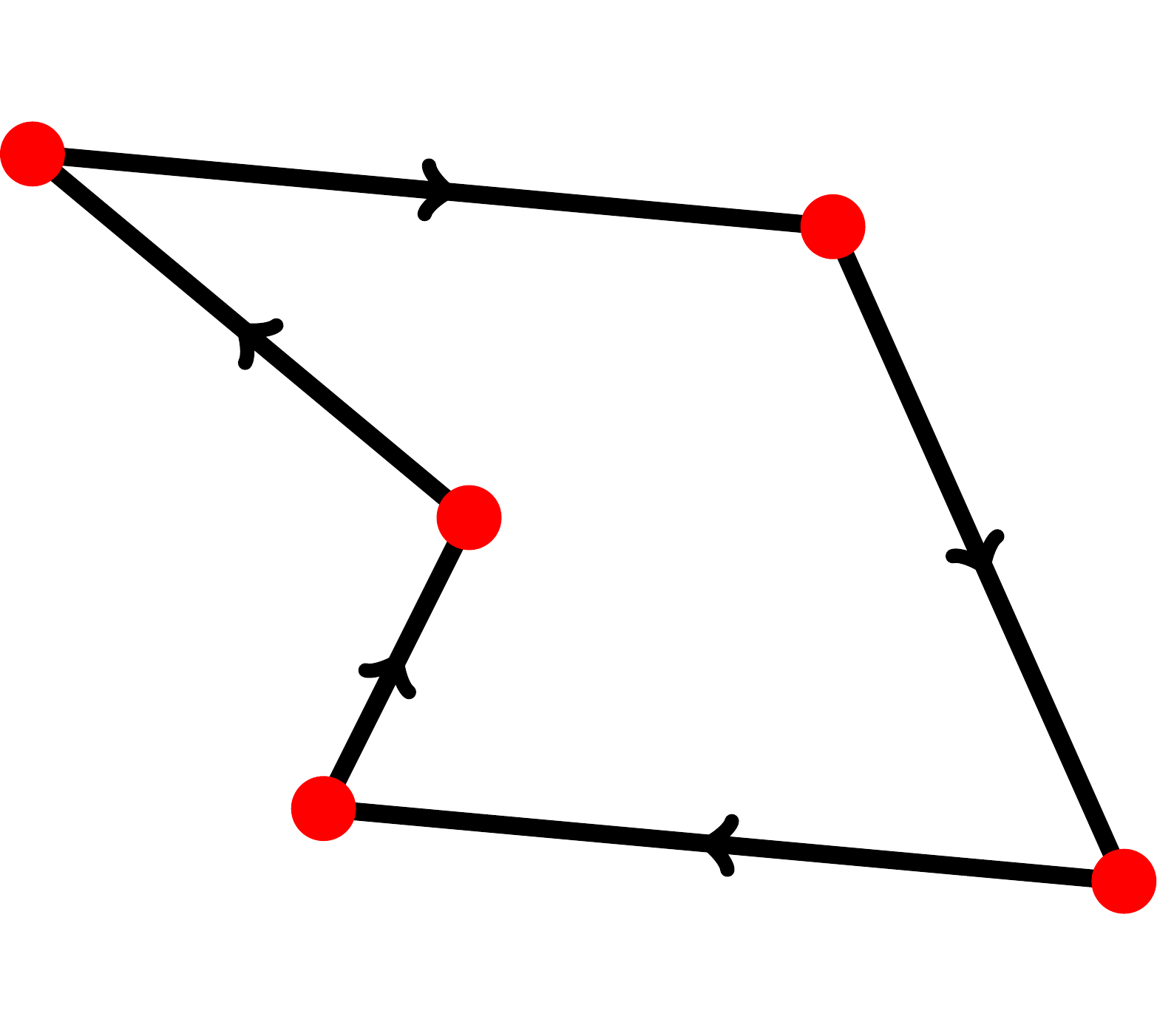}
        \hfill
        \caption{}
        \label{fig:sequence:planner}
    \end{subfigure}%
    \begin{subfigure}{0.33\linewidth}
        \centering
        \includegraphics[width=0.8\linewidth]{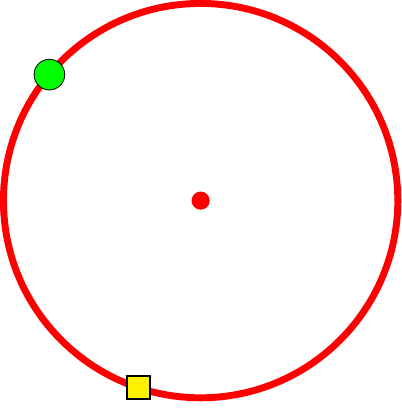}
        \caption{}
        \label{fig:coordinator}
    \end{subfigure}%
    \begin{subfigure}{0.33\linewidth}
        \centering
        \includegraphics[width=0.8\linewidth]{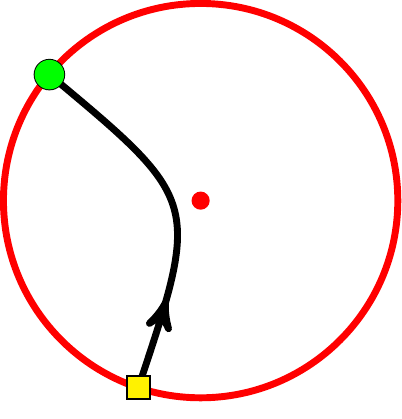}
        \caption{}
        \label{fig:local:OCP}
    \end{subfigure}%
    \caption{Illustrating the three optimization problems of the decomposition: sequence planning; entrance and departure point optimization; and local trajectory optimization.}
    \label{fig:graphical:hierarchies}
\end{figure}

The optimization problems within the individual levels are depicted in \figref{fig:graphical:hierarchies}: \ref{fig:sequence:planner} illustrates the problem of finding a target visiting sequence; \ref{fig:coordinator} illustrates the problem of optimizing the entrance (yellow square) and departure (green circle) points when visiting a target; and \ref{fig:local:OCP} depicts a local \ac{ocp}, which determines the trajectory that drives the target uncertainty to zero during the visit. Note that the red circle in~\ref{fig:coordinator} and~\ref{fig:local:OCP} depicts the sensing radius around the considered target. 

In this paper we focus on the optimization of the entrance and departure points together with the low-level trajectories. However, we remark that visiting sequences can be computed using a graph-based abstraction, e.g., via \ac{tsp} or other methods specifically designed for \ac{pm} problems~\cite{welikala2021event,hari2020optimal}. We further discuss this along with additional extensions in \secref{sec:conclusion:future:work}. 

\textbf{Coordinating the local trajectories.} 
From here on we assume that a periodic target visiting sequence $i_1, i_2, \dots i_K, i_1$ is provided. The coordinator is tasked to realize the visiting sequence and coordinate the requirement of driving the target uncertainties down to zero. To do this, we note that during the $k$th target visit, the agent begins sensing the target $i_k$ at a specific point in space, which we denote by $s_k^\varphi$ and refer to as the entrance point (yellow square in~\ref{fig:coordinator}). Similarly, there exists a departure point $s_k^\psi$ (green circle in~\ref{fig:coordinator}), i.e., a point at which the agent last sensed the target. The coordinator passes the entrance and departure points down to the local \ac{ocp} solver generating the local trajectory segments as discussed at the end of this section. In return, the local \acp{ocp} provide dual variables specifying the cost associated with the entrance and departure constraints. The coordinator then utilizes these dual variables to optimize the entrance and departure points with the goal of minimizing the total cycle time (see \secref{sec:solving:coordinator}).

\figref{fig:hierarchies} depicts the proposed workflow of the coordinator (dashed box). It receives an initial guess of the entrance and departure points generated from the visiting sequence (the generation is discussed in \secref{sec:solving:coordinator}). The coordinator then calls the local \ac{ocp} solver in an event-based fashion, i.e., whenever the agent starts or stops sensing a target. On completion of a cycle, the coordinator updates the entrance/departure points as well as the target uncertainties at cycle start, or terminates the algorithm if the cost gradients with respect to the entrance and departure points are sufficiently small and steady-state is reached. 

\begin{figure}
    \centering
    \includegraphics[width=0.75\linewidth]{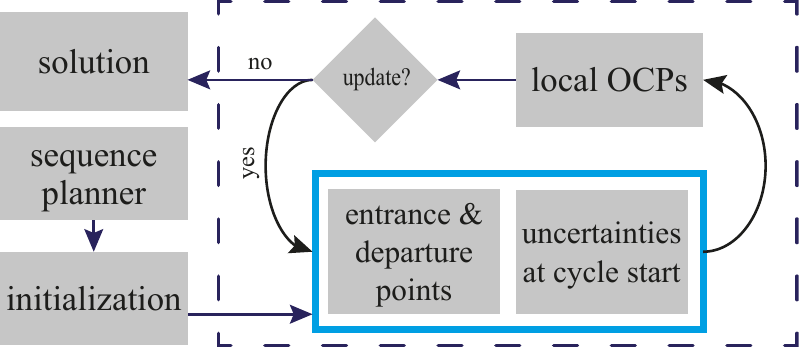}
    \caption{Illustrating the workflow of the coordinator, which is part of the agent's control system.}
    \label{fig:hierarchies}
\end{figure}

\textbf{Formulating the local OCPs.} 
There are two types of local \acp{ocp} to be solved: that of driving a target's uncertainty to zero (the \textit{draining problem}) and that of moving from the departure point $s_k^\psi$ of one target to the entrance point of the next target $s_{k+1}^\varphi$ (the \textit{switching problem}). 

Given an unconstrained environment and first-order dynamics, the switching problem becomes trivial as it is given by a maximal and constant control input that moves the agent from one point to another along a straight line. However, the current formulation is capable of adapting to other scenarios. For example, if obstacles are present in the environment then the problem becomes more complicated but can often still be achieved using  optimal control techniques~\cite{beaver2023IFAC}. We solely require the local switching problem to provide cost sensitivities with respect to the constraints that fix the boundary conditions $s_k^\psi$ and $s_{k+1}^\varphi$.

Let us now focus on the draining problem, which consists of finding a time optimal trajectory that drives the target uncertainty down to zero while satisfying the constraints imposed by the coordinator. Specifically, this is given by the \ac{ocp}
\begin{mini!}
    {u(\cdot), T_k}{\int_0^{T_k} \dt}{\label{ocp:local:first}}{}
    \addConstraint{\dot{s}(t)}{= u(t)}
    \addConstraint{\dot{R}_{i_k}(t)}{= f_{R_{i_k}}(R_{i_k}(t), s(t)) \label{ocp:local:first:R}}
    \addConstraint{\|u(t)\|^2}{\leq 1}
    \addConstraint{\min_{\tau \in [0, T_k]} R_{i_k}(\tau)}{= 0 \label{ocp:local:first:draining}}
    \addConstraint{s(0)}{= s_k^\varphi}
    \addConstraint{s(T_k)}{= s_k^\psi \label{ocp:local:first:exit}}
    \addConstraint{R_{i_k}(0)}{= \check{R}_k,}
\end{mini!}    
where $s_k^\varphi$ and $s_k^\psi$ denote the respective entrance and departure points, and $\check{R}_k$ denotes the uncertainty at arrival time, all of which are passed down from the coordinator. At first glance, this problem seems challenging to solve due to the non-smoothness of $f_{R_{i_k}}$ and the unconventional constraint~\eqref{ocp:local:first:draining}, which enforces the uncertainty to be drained. In the next section we discuss how the problem can be solved efficiently.

\section{Solving the Draining OCP}\label{sec:solution:local:ocp}
In this section we devote our attention to the local draining \ocpref{ocp:local:first}. We introduce the following results as building blocks to \theoref{theorem:exact:reformulation}, which reformulates \eqref{ocp:local:first} into a smooth \ac{ocp}.

\begin{lemma}\label{lemma:straight:line:out}
    Consider any optimal trajectory of~\ocpref{ocp:local:first} with optimal cost $T_k^\ast$. Further, let $t_0 \in (0, T_k^\ast)$ be any time such that $R_{i_k}^\ast(t_0) = 0$. Then
    \begin{equation}\label{eq:optimal:control:out}
        u^\ast(t) = \frac{s_k^\psi - s^\ast(t_0)}{\|s_k^\psi - s^\ast(t_0)\|}
    \end{equation}
    for all $t \in (t_0, T_k^\ast)$.
\end{lemma}
\begin{proof}
    Let $(s^\ast, R_{i_k}^\ast)$ be any optimal trajectory of~\eqref{ocp:local:first} with optimal control $u^\ast$, and let $t_0$ be a time with $R_{i_k}^\ast(t_0)$. If $u^\ast(t)$ differs from~\eqref{eq:optimal:control:out} over a nonempty open interval $\I \subset [t_0, T_k^\ast]$, then $\|s_k^\psi - s^\ast(t_0)\| < T_k^\ast - t_0$. Further, the control law
    \begin{equation*}
        \tilde{u}(t) = \begin{cases}
            u^\ast(t), & \ift t \leq t_0, \\
            \frac{s_k^\psi - s^\ast(t_0)}{\| s_k^\psi - s^\ast(t_0) \|}, & \othert,
        \end{cases}
    \end{equation*}
    is feasible and leads to a strictly lower cost, which contradicts the optimality assumption of $s^\ast$.
\end{proof}

Although \lemref{lemma:straight:line:out} characterizes the final piece of an optimal trajectory, it does not help identifying an adequate time $t_0$ or position $s^\ast(t_0)$. \theoref{theorem:unique:zero:level:set} will characterize the point $s^\ast(t_0)$. However, we require a few more lemmas, including the next technical statement, which captures the fact that any homotopy of feasible control laws is a feasible control law itself.
\begin{lemma}\label{lemma:homotopy}
    Let $u_1 : [0,T_1] \to \R^2$ and $u_2 : [0, T_2] \to \R^2$ be two feasible control laws for \ocpref{ocp:local:first}, and define $T_\sigma = (1-\sigma)T_1 + \sigma T_2$ for any homotopy parameter $\sigma \in [0,1]$. Then, the control law $u_\sigma : [0, T_\sigma] \to \R^2$ defined by
    \begin{equation*}            
        u_\sigma(t) = (1-\sigma) u_1\left( \frac{t}{T_\sigma} T_1\right) + \sigma u_2 \left( \frac{t}{T_\sigma} T_2\right)
    \end{equation*}
    satisfies $\| u_\sigma(t) \| \leq 1$ and generates a trajectory with $s_\sigma(0) = s_k^\varphi$ and $s_\sigma(T_\sigma) = s_k^\psi$.
\end{lemma}
\begin{proof}
    The former statement follows using the triangle inequality and the latter statement by integrating $u_\sigma(t)$ over $t$ from $0$ to $T_\sigma$.
\end{proof}

\begin{lemma}\label{lemma:0:straight:line}
    Let $u^\ast$ be an optimal control of~\eqref{ocp:local:first} and assume that there exists a nonempty open interval $\I = (t_1, t_2)$ such that $R_{i_k}^\ast(t) = 0$ on $\I$. Then
    \begin{equation}\label{eq:optimal:control:in:0}
        u^\ast(t) = \frac{s^\ast(t_2) - s_k^\varphi}{t_2}
    \end{equation}
    on $[0, t_2]$.
\end{lemma}
\begin{proof}
    Assume that $u^\ast$, $R_{i_k}^\ast$ and $\I$ are as above and let $T_k^\ast$ be the associated optimal cost. Further, assume that $u^\ast$ differs from~\eqref{eq:optimal:control:in:0} over a nonempty open interval. Define $\tilde{T} = t_2 + \| s_k^\psi - s^\ast(t_2) \|$ and consider the control policy 
    \begin{equation*}
        \tilde{u}(t) = \begin{cases}
            \frac{s^\ast(t_2) - s_k^\varphi}{t_2}, & \ift 0 \leq t < t_2, \\
            \frac{s_k^\psi - s^\ast(t_2)}{\tilde{T} - t_2}, & \ift t_2 \leq t \leq \tilde{T}_k,
        \end{cases}
    \end{equation*}
    which satisfies the boundary conditions but may not drain the uncertainty. Now let $\sigma \in (0,1)$ be a homotopy parameter. Then
    \begin{equation}
        u_\sigma(t) = (1 - \sigma)u^\ast(t) + \sigma\tilde{u}(t)
    \end{equation}
    satisfies the boundary conditions of~\eqref{ocp:local:first} due to \lemref{lemma:homotopy}. Furthermore, $u_\sigma$ satisfies
    \begin{equation*}
        \lim_{\sigma \to 0} u_\sigma(t) = u^\ast(t),
    \end{equation*}
    showing that $u_\sigma$ converges to a feasible control law. Now let $R_{i_k}^\sigma$ denote the uncertainty trajectory under the control $u_\sigma$. Then 
    \begin{equation}
        \lim_{\sigma \to 0} R_{i_k}^\sigma(t) = R_{i_k}^\ast(t).
    \end{equation}  
    Due to continuity of $R_{i_k}^\ast$ together with the fact that $R_{i_k}^\ast(t) = 0$ on $\I$ imply that there exists an $\varepsilon > 0$ such that $R_{i_k}^\varepsilon(t) = 0$ for some $t \in [0, T_\varepsilon]$, where $T_\varepsilon = (1-\varepsilon) T_k^\ast + \varepsilon \tilde{T}$. But $T_\varepsilon < T_k^\ast$ with $u_\varepsilon$ feasible contradict the optimality assumption of $u^\ast$.
\end{proof}

\begin{corollary}\label{corollary:0:straight:line}
    Let $u^\ast$ be an optimal control of~\eqref{ocp:local:first} and assume that there exists a nonempty open interval $\I = (t_1, t_2)$ such that $R_{i_k}^\ast(t) = 0$ on $\I$. Then $T_k^\ast = \| s_k^\psi - s_k^\varphi \|$ and 
    \begin{equation}
        u^\ast(t) = \frac{s_k^\psi - s_k^\varphi}{T_k^\ast}
    \end{equation}
    for all $t \in [0, T_k^\ast]$.
\end{corollary}
\begin{proof}
    \lemref{lemma:0:straight:line} shows that $u^\ast(t)$ is constant on $[0, t_2]$ while~\lemref{lemma:straight:line:out} shows that $u^\ast(t)$ is constant on $[t_1, T_k^\ast]$. Note that $t_1 < t_2$ which implies that the two controls must be identical. The claim follows.
\end{proof}

\begin{figure}
    \centering
    \includegraphics[width=0.5\linewidth]{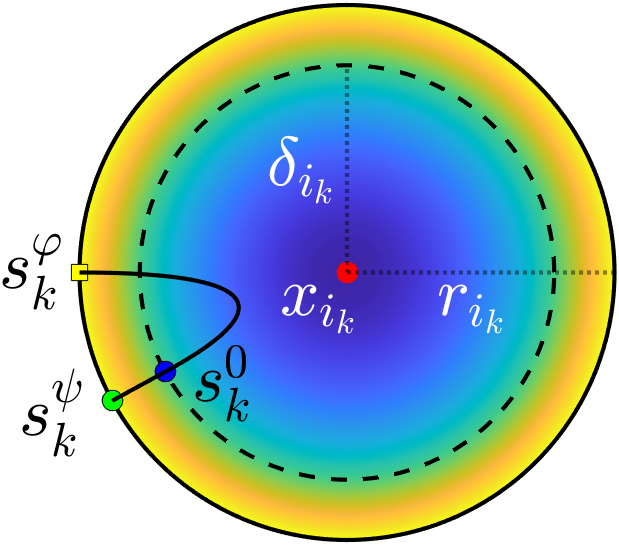}
    \caption{Illustrating a solution of the draining \ac{ocp}~\eqref{ocp:local:first} of target $x_{i_k}$ with entrance point $s_k^\varphi$, outer departure point $s_k^\psi$ and inner departure point $s_k^0$ (cf.~\eqref{ocp:local}). The function $f_{R_{i_k}}$ takes positive values when the agent-target distance exceeds $\delta_{i_k}$, whereas it is negative in the interior of the inner circle (dashed).}
    \label{fig:uncertainty}
\end{figure}

We now utilize \lemref{lemma:straight:line:out} in order to characterize a point on the agent's trajectory which allows the verification of the draining condition~\eqref{ocp:local:first:draining}. Uniqueness will follow from \corref{corollary:0:straight:line}.
\begin{theorem}\label{theorem:unique:zero:level:set}
    For any optimal trajectory of~\eqref{ocp:local:first} there exists a unique time $t_k^0$ such that $\| \agent^\ast(t_k^0) - x_{i_k} \| = \delta_{i_k}$ and $R_{i_k}^\ast(t_k^0) = 0$. Further, it holds 
    \begin{equation}\label{eq:optimal:control:out:2}
        u^\ast(t) = \frac{s_k^\psi - \agent^\ast(t_k^0)}{\| s_k^\psi - \agent^\ast(t_k^0) \|}
    \end{equation}
    for every $t \in [t_k^0, T_k^\ast]$.
\end{theorem}
\begin{proof}
    Consider \figref{fig:uncertainty} and note that in order to drive the target's uncertainty to zero, i.e., satisfy~\eqref{ocp:local:first:draining}, it is necessary for the agent to enter the inner circle of radius $\delta_{i_k} = \sqrt{r_{i_k}^2(B_{i_k}-A_{i_k})/B_{i_k}}$ and remain in its interior until the uncertainty is drained at some time $t_0$. Applying \lemref{lemma:straight:line:out} implies~\eqref{eq:optimal:control:out:2}. The resulting trajectory from $s^\ast(t_0)$ to $s^\ast(T_k)$ forms a straight line, and thus intersects the zero level set of the uncertainty dynamics at a time $t_k^0 \in [t_0, T_k)$. Note that $R_{i_k}^\ast(t_k^0) = 0$ by construction. 
    
    Now assume the point were not unique. Then, necessarily $t_0 < t_k^0$ and $\| s^\ast(t_0) - x_{i_k} \| = \delta_{i_k}$. The straight line trajectory from $s^\ast(t_0)$ to $s^\ast(t_k^0)$ would lie in the interior of the inner circle and consequently $R_{i_k}(t) = 0$ for all $t \in (t_0, t_k^0)$. Applying \corref{corollary:0:straight:line} would imply that the entire trajectory forms a straight line. This line would intersect the inner circle twice. But on the first such intersection it is impossible for the uncertainty to vanish, since its dynamics were nonnegative up to that point. This contradicts $R_{i_k}(t_0) = 0$, implying uniqueness.
\end{proof}

A special case of the optimal control occurs whenever the initial uncertainty $\check{R}_{i_k}$ is large enough such that the greedy control policy~\eqref{eq:greedy:control} becomes optimal. This result is provided in \propref{proposition:greedy:optimal}. Let us first prove a lemma for this statement. In the following statements we drop the indices of $i_k$ from the variables $A,B,r$ and $\delta$ in order to increase readability.

\begin{lemma}\label{lemma:greedy:control}
    If
    \begin{equation}
        \check{R}_{i_k} \geq -\left(A - \tfrac{2}{3}B\right)r - \delta(A - B) - \tfrac{\delta^3}{3 r^3}B
    \end{equation}
    then the optimal cost $T_k^\ast$ of~\eqref{ocp:local:first} is bounded below by $2r$.
\end{lemma}
\begin{proof}
    We prove this statement via contraposition. Assume there exists an optimal trajectory $s^\ast$ with cost $T_k^\ast < 2r$. Let $t_k^0$ be the unique inner exit time as to~\propref{theorem:unique:zero:level:set}. It immediately follows that $t_k^0 < 2r - (r-\delta) = r + \delta$. Now introduce the greedy control law 
    \begin{equation}\label{eq:greedy:control}
        \tilde{u}(t) = \begin{cases}
            \frac{x_{i_k} - s_k^\varphi}{r}, &\ift 0 \leq t < r, \\
            0, &\ift r \leq t < T_k^\ast-r, \\
            \frac{s_k^\psi - x_{i_k}}{r}, &\ift T_k^\ast -r \leq t \leq T_k^\ast.
        \end{cases}
    \end{equation}
    This control law leads to an infeasible trajectory $\tilde{s}(t)$ if $T_k^\ast < 2r$, as it does not reach the departure point $s_k^\psi$. However, $\tilde{s}(t)$ satisfies $\| \tilde{s}(t)  - x_{i_k} \| \leq \| s^\ast(t) - x_{i_k} \|$ for all $t \in [0, T_k^\ast]$ and thus provides
    \begin{equation*}
        \begin{split}
            R_{i_k}^\ast(t_k^0) &= \check{R}_{i_k} + \int_0^{t_k^0} \dot{R}_{i_k}^\ast(t) dt 
            \geq \check{R}_{i_k} + \int_0^{t_k^0} \dot{\tilde{R}}_{i_k}(t) dt \\
            &> \check{R}_{i_k} + \int_0^{r+\delta} \dot{\tilde{R}}_{i_k}(t) dt \\
            &= \check{R}_{i_k} + \left(A - \tfrac{2}{3}B\right)r + (A - B)\delta + \tfrac{\delta^3}{3 r^3}B.
        \end{split}
    \end{equation*}
    Knowing $R_{i_k}^\ast(t_k^0) = 0$ concludes the proof.
\end{proof}

\begin{proposition}\label{proposition:greedy:optimal}
    If
    \begin{equation}\label{proposition:greedy:optimal:bound}
        \check{R}_{i_k} \geq -\left(A - \tfrac{2}{3}B\right)r - \delta(A - B) - \tfrac{\delta^3}{3 r^3}B
    \end{equation}
    then the greedy control policy~\eqref{eq:greedy:control} is an optimal control of~\eqref{ocp:local:first}, with optimal cost 
    \begin{equation*}
        T_k^\ast = \frac{-\check{R}_{i_k} - \bigl(A - \tfrac{2}{3}B\bigr)r - \delta(A - B) - \tfrac{\delta^3}{3 r^3}B}{A-B} + 2r.
    \end{equation*}
\end{proposition}
\begin{proof}
    \lemref{lemma:greedy:control} shows that the optimal cost $T_k^\ast$ of the local \ocpref{ocp:local:first} is bounded below by $2r$, which shows that the greedy policy achieves $s^\ast(T_k^\ast) = s_k^\psi$, in other words becomes feasible. Any other control policy connecting $s_k^\varphi$ with $s_k^\psi$ in $T_k^\ast$ units of time satisfies $\| s^\ast(t) - x_{i_k}\| \leq \| s(t) - x_{i_k} \|$ for all $t$, which implies $R_{i_k}^\ast(t) \leq R_{i_k}(t)$ for all $t \in [0,T_k^\ast]$. Further, a straightforward computation shows that $R_{i_k}^\ast(t_k^0)$, where $t_k^0 = T_k^\ast - (r-\delta)$, is given by 
    \begin{equation*}
        \check{R}_{i_k} + \left(A - \tfrac{2}{3}B\right)r + \delta(A - B) + \tfrac{\delta^3}{3 r^3}B + (T_k^\ast - 2r)(A - B).
    \end{equation*}
    Plugging in $T_k^\ast$ shows $R_{i_k}^\ast(t_k^0) = 0$. 
\end{proof}

Motivated by the above characterization we reformulate~\eqref{ocp:local:first} into the smooth \ac{ocp}
\begin{mini!}
    {u(\cdot), t_k^0, s_k^0}{\int_0^{t_k^0} \dt + \| s_k^\psi - s_k^0 \| }{\label{ocp:local}}{}
    \addConstraint{\dot{s}(t)}{= u(t)}
    \addConstraint{\dot{R}_{i_k}(t)}{= A_{i_k} - B_{i_k} p_{i_k}(s(t)) \label{ocp:local:R}}
    \addConstraint{\|u(t)\|^2}{\leq 1}
    \addConstraint{R_{i_k}(t_k^0)}{\leq 0 \label{ocp:local:draining}}
    \addConstraint{s(0)}{= s_k^\varphi \label{ocp:local:phi}}
    \addConstraint{s(t_k^0)}{= s_k^0 \label{ocp:local:psi}}
    \addConstraint{\|s_k^0 - x_{i_k}\|}{= \delta_{i_k} \label{ocp:zero:level:set:constraint}}
    \addConstraint{R_{i_k}(0)}{= \check{R}_k. \label{ocp:local:Rcheck}}
\end{mini!}

\begin{theorem}\label{theorem:exact:reformulation}
    The relaxation~\eqref{ocp:local} is exact, i.e., any optimal trajectory $s^\ast$ of \ocpref{ocp:local} is also optimal for~\eqref{ocp:local:first}. Further, the respective uncertainty trajectory can be recovered from the relaxed counterpart.
\end{theorem}
\begin{proof}
    The given reformulation adopts a few changes. First, we replaced the original terminal condition with an adapted terminal condition plus a terminal cost. The adapted terminal condition now specifies a point on the zero level set of the dynamics, making it easy to check the draining condition. The terminal cost reflects precisely the amount of time it takes to travel from the adapted terminal point to the original terminal point. This part of the reformulation is exact due to \theoref{theorem:unique:zero:level:set}. Additionally, the reformulation replaced the hybrid dynamics~\eqref{ocp:local:first:R} by the smooth dynamics~\eqref{ocp:local:R}. This may cause the uncertainty $R_{i_k}$ to become negative. However, it is easy to see that the uncertainty of the original problem is $0$ if and only if the uncertainty of the reformulation is non-positive, as long as the agent trajectory does not leave and reenter the inner circle. Since this is the case for any optimal trajectory, we find that the inequality constraint~\eqref{ocp:local:draining} together with the relaxation~\eqref{ocp:local:R} reflect the original draining condition.

    Given a solution for the relaxed problem, we can reconstruct the true uncertainty trajectory by projecting its negative segment to zero, until the uncertainty gradient becomes nonnegative and then shifting the remaining trajectory piece by -$R_{i_k}(t_k^0)$ (see \figref{fig:relaxation:recovery}).
    \begin{figure}
        \centering
        \includegraphics[width=\linewidth]{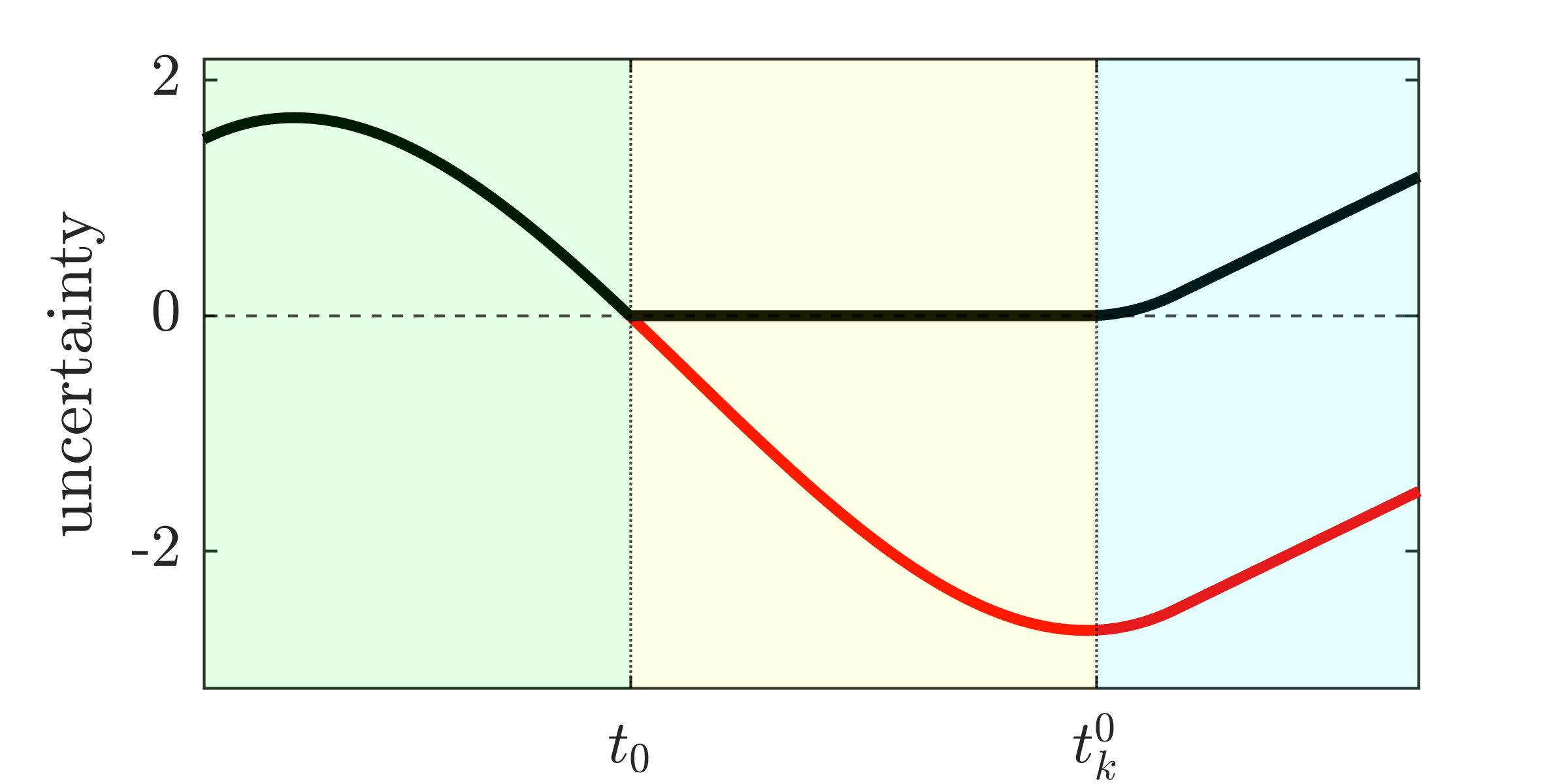}
        \caption{Projecting the relaxation (red) to zero during the yellow interval, and shifting it during the blue interval by the violation at $t_k^0$ recovers the true uncertainty trajectory (black).}
        \label{fig:relaxation:recovery}
    \end{figure}
\end{proof}

\section{Optimizing the Local Boundary Conditions}\label{sec:solving:coordinator}
The goal of this section is to minimize the cycle time $T$ given a sequence of target visits $i_1,i_2,\dots,i_K \in \T$. By decomposing the cycle into its local segments, we can express the total cycle time in terms of the entrance and departure points, i.e., as
\begin{equation*}
    T = \sum_{k=1}^K T_k^\ast(s_k^\varphi, s_k^\psi) + \Delta_k^\ast(s_k^\psi, s_{k+1}^\varphi),
\end{equation*}
where $T_k^\ast$ denotes the time of the $k$th draining trajectory, and $\Delta_k^\ast = \| s_k^\psi - s_{k+1}^\varphi \|$ denotes the $k$th switching time, i.e., the time taken from the departure point $s_k^\psi$ to the next entrance point $s_{k+1}^\varphi$. Now note that the entrance and departure points can be expressed as
\begin{equation*}
    s_k^\varphi = x_{i_k} + r_{i_k} \begin{pmatrix}
        \cos(\varphi_k) \\ \sin(\varphi_k)
    \end{pmatrix} ~ \andt ~ s_k^\psi = x_{i_k} + \delta_{i_k} \begin{pmatrix}
        \cos(\psi_k) \\ \sin(\psi_k)
    \end{pmatrix},
\end{equation*}
respectively, where we choose to optimize the inner departure point instead of the outer departure point motivated by \theoref{theorem:exact:reformulation}. Then the only degrees of freedom are the angles $\varphi_k$ and $\psi_k$. From here on we denote by $\varphi, \psi \in \R^{K}$ the vectors that contain all entrance/departure angles. We may then express the total cycle time $T(\varphi, \psi)$ as a function of these parameters. We are left to solve the unconstrained bilevel minimization problem
\begin{mini}
    {\varphi, \psi}{\sum_{k=1}^{K} T_k^\ast(\varphi_k, \psi_k) + \Delta_k^\ast(\psi_k, \varphi_{k+1}).}{\label{mini:global}}{}
\end{mini}
Solving this problem can be done on-line in the following manner. We first initialize the entrance and departure points using
\begin{equation}\label{eq:angles:initialization}
    \begin{split}    
    \psi_k = \mathrm{atan2}(\vartheta^k_y, \vartheta^k_x),
    \end{split},~\varphi_{k+1} = \mathrm{atan2}(-\vartheta^{k}_y, -\vartheta^{k}_x)
\end{equation} 
where $\vartheta^k = x_{k+1}-x_{k}$. This natural initialization places $s_k^\psi$ and $s_{k+1}^\varphi$ on the straight lines between the agents $i_k$ and $i_{k+1}$. Further, we and set the initial uncertainty vector to zero, or to an estimate of the steady-state values. We place the agent at the exit point of target $i_K$ and apply the constant control $(s_1^\varphi - s_K^\psi)/\Delta_K$. When the agent arrives at the first entrance point, we solve a discretized version of the smooth draining \ocpref{ocp:local} for the first target (the discretization is discussed in \secref{sec:numerical:results}). This provides an open loop control law during the draining period, or alternatively we may choose to apply a closed loop controller to track the computed trajectory. When reaching the departure point, we repeat the process for the next target. 

On completion of a cycle we compute the gradients 
\begin{equation}\label{eq:global:gradients}
    \begin{split}
        \frac{\partial T}{\partial \varphi_k} = \frac{\partial T_k^\ast}{\partial \varphi_k} + \frac{\partial \Delta^\ast_{k-1}}{\partial \varphi_k},\quad 
        \frac{\partial T}{\partial \psi_k} = \frac{\partial T_k^\ast}{\partial \psi_k} + \frac{\partial \Delta^\ast_{k}}{\partial \psi_k}, 
    \end{split}
\end{equation}
of the cycle time $T(\varphi, \psi)$. Note that gradients of $\Delta^\ast$ can be computed analytically, while evaluation of the gradients of $T_k^\ast$ can be done by using the dual variables, or Lagrange multipliers~\cite{nocedal1999numerical}, of the respective constraints. Let us denote by $\lambda_{\varphi_k}, \lambda_{\psi_k} \in \R^2$ the dual variables of the entrance~\eqref{ocp:local:phi} and departure constraints constraint~\eqref{ocp:local:psi} of the $k$th local \ac{ocp}, respectively. Applying the chain rule then provides 
\begin{equation*}
    \begin{split}
        \frac{\partial T_k^\ast}{\partial \varphi_k} &= \lambda_{\varphi_k}^\top r_{i_k} \begin{pmatrix} -\sin(\varphi_k) \\ \cos(\varphi_k) \end{pmatrix}, \\
        \frac{\partial T_k^\ast}{\partial \psi_k} &= \lambda_{\psi_k}^\top \delta_{i_k} \begin{pmatrix} -\sin(\psi_k) \\ \cos(\psi_k) \end{pmatrix}.  
    \end{split}
\end{equation*}
We then update the parameters using a simple gradient descent law
\begin{equation}
    \begin{split}
        \varphi_k \leftarrow \varphi_k - \alpha \frac{\partial T}{\partial \varphi_k}, \quad
        \psi_k \leftarrow \psi_k - \alpha \frac{\partial T}{\partial \psi_k},
    \end{split}
\end{equation}
where $\alpha$ is chosen using a diminishing step-size rule.

\section{Numerical Results}\label{sec:numerical:results}

\begin{figure*}
    \includegraphics[width=0.32\linewidth]{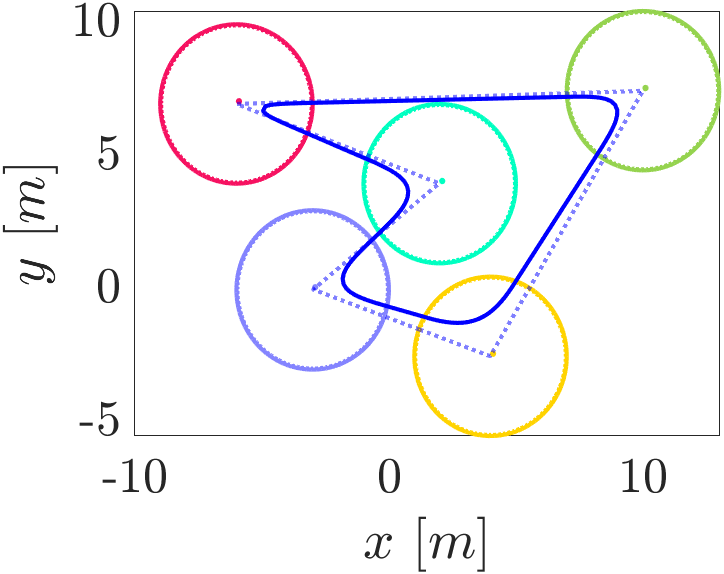}\hfill
    \includegraphics[width=0.32\linewidth]{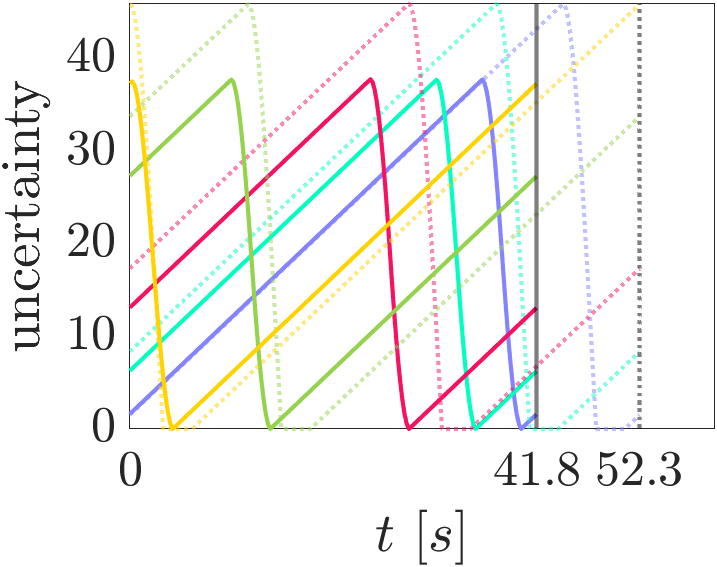}\hfill%
    \includegraphics[width=0.31\linewidth]{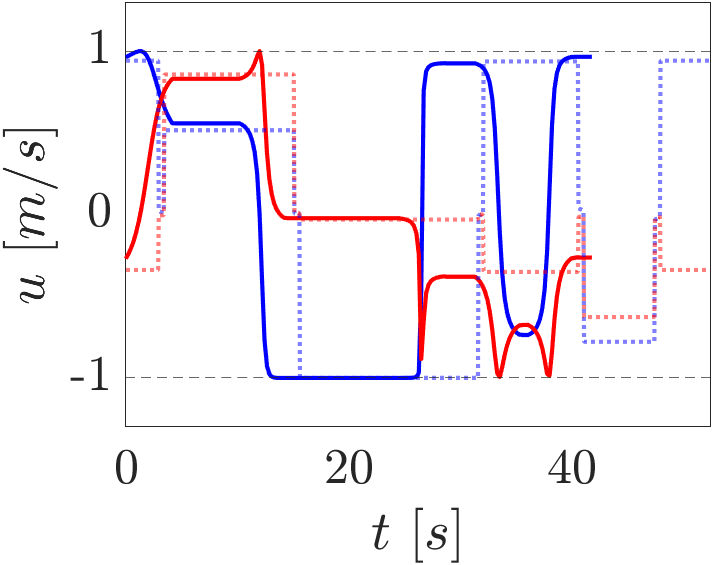}%
    \caption{Comparing the greedy solution (dotted) to the solution obtained via the proposed method (solid).}
    \label{fig:experiment:greedy}
\end{figure*}

Motivated by the fact that greedy policies are able to produce optimal solutions under certain scenarios (see~\lemref{lemma:greedy:control}), we compare the proposed method to the greedy control policy: move towards the target (and potentially dwell there) until uncertainty hits $0$, then move to the next target. We consider homogeneous targets with $A_i = 1$, $B_i = 20$, $r_i = 3$, and $R_i(0) = 0$. 

We discretize the local OCPs via direct multiple shooting~\cite{rao2009survey} using explicit Euler integration over $20$ nodes. We model this in Matlab via CasADi~\cite{andersson2019casadi} and then solve the resulting nonlinear programs using IPOPT~\cite{wachter2006implementation}. The underlying hardware consists of an Intel i5 processor running at 1.60GHz with 16GB of RAM. 

 \figref{fig:experiment:greedy} shows that even though the trajectories appear to be similar, our approach shows a 20\% reduction in total travel time over the greedy policy, as can be seen on the right plot depicting the respectively obtained steady-state cycles with periods of 41.8 and 52.3. Furthermore, the proposed method leads to a smooth control profile, which is favorable for tracking feasibility. We conjecture that this smoothness is a fundamental behavior for solutions of~\eqref{mini:global}, with exceptions being settings where the optimal entrance and departure points coincide, or settings where the the initial uncertainty is large enough to satisfy~\eqref{proposition:greedy:optimal:bound}. The following intuition justifies this conjecture: if the entrance (or departure) transition is non-smooth, then this indicates that the angle between the entrance and departure points is too large. Reducing the angle between the entrance and departure point always reduces the local draining time $T_k^\ast$. The only thing that prevents the entrance and departure points from converging to each other is the trade-off introduced by the potentially increased switching times $\Delta_{k-1}$ and $\Delta_{k}$. The equilibrium of this trade-off, namely the solution of~\eqref{mini:global}, may thus induce a natural property of smooth transitions. The only non-smoothness along trajectories could then occur within the local draining trajectory, which should only occur in the exceptions discussed above. 
 
 We now shift our attention to the computational efficiency of the proposed method. The left plot in \figref{fig:optimization:statistics} shows the CPU timings for this experiment, where we recall that the draining \ac{ocp} refers to~\eqref{ocp:local} and the switching \ac{ocp} refers to the (trivial) problem of switching from target's departure point to the next target's entrance point. The relaxation of the hybrid dynamics as well as the reduction of the state space dimensionality lead to trajectory segments computed in fractions of a second, suggesting real-time feasibility for systems with update rates of 50-80~Hz. The central plot in \figref{fig:optimization:statistics} shows that the method converges within 25 cycles to a steady-state and optimal solution. We conduct one final experiment by randomly initializing the entrance and departure points with angles drawn from a uniform distribution in the range of $[0,2\pi]$ over 100 experiments. The resulting transient cycle times all converge to the same steady-state solution as visualized in \figref{fig:optimization:statistics}. This figure suggests yet another conjecture: the problem of optimizing the entrance and departure points is convex. 

\begin{figure*}
    \includegraphics[width=0.32\linewidth]{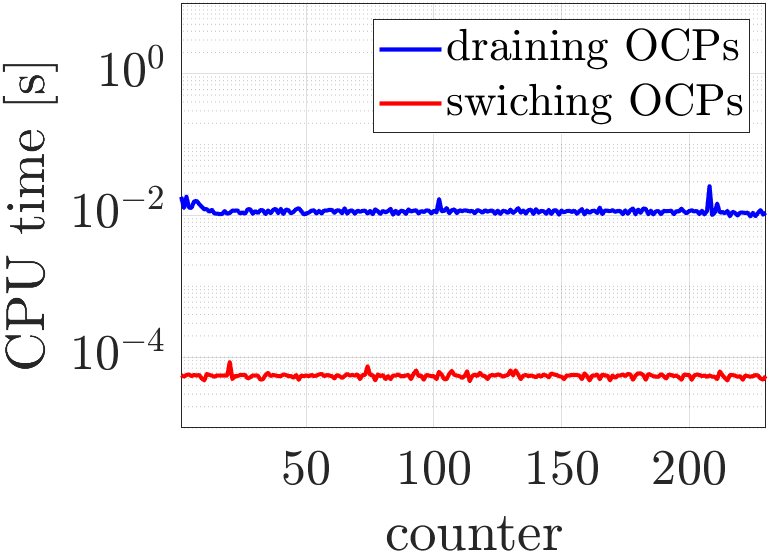}\hfill
    \includegraphics[width=0.31\linewidth]{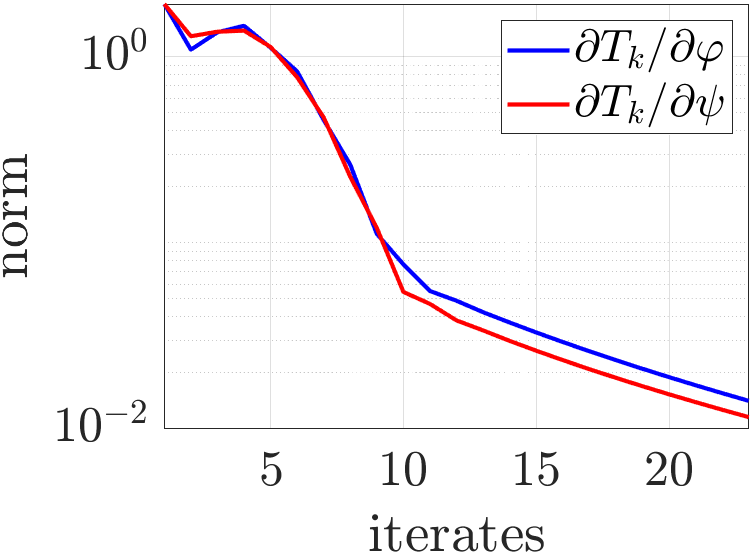}\hfill
    \includegraphics[width=0.30\linewidth]{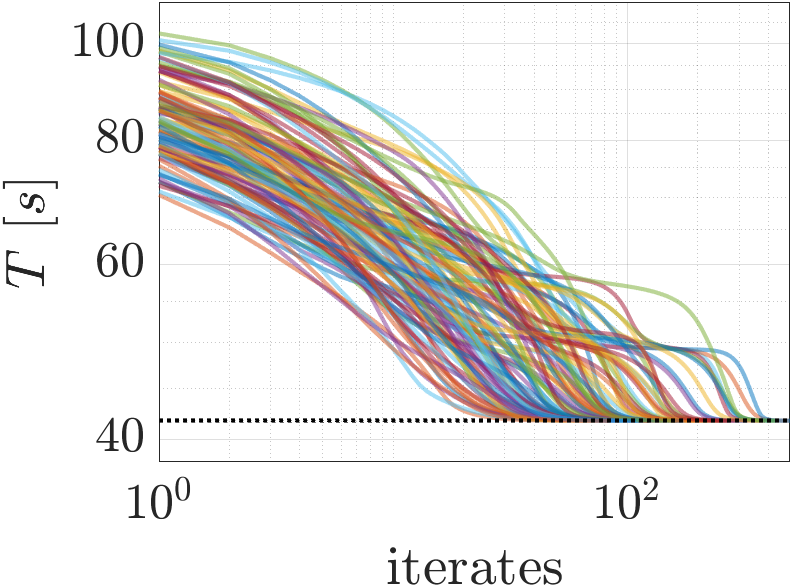}\hfill
    \caption{Depicting the CPU times required for solving the local OCPs (left), the convergence of the gradient norm (middle) and the convergences of 100 randomly perturbed initializations to the same optimizers (right).}
    \label{fig:optimization:statistics}
\end{figure*}

\section{Conclusion and Future Work}\label{sec:conclusion:future:work}
In this paper we considered a two-dimensional infinite-horizon \ac{pm} problem, in which we are interested in finding a minimum time draining cycle. By decomposing the problem into local OCPs on the lowest level, and coordinating their trajectories via higher level parameters, we were able to prove the existence of an exact relaxation for the underlying hybrid dynamics. These two layers are coupled within a bilevel optimization scheme, with which the agent's trajectory is optimized on-line. 

In future work we aim at analyzing the two conjectures made in \secref{sec:numerical:results}, as well as extending to scenarios to three dimensions, or to multi-agent settings. Furthermore, a particularly interesting extension could arise when the draining condition was relaxed. This could be done by introducing new parameters that specify the right hand side of~\eqref{ocp:local:first:draining}, which could then be optimized in a similar way as was done here with the entrance and departure points. Apart from extending the introduced approach, we also desire comparing the proposed approach to existing methods, e.g., by simply utilizing the general optimal control problem solver~\cite{nurkanovic2022nosnoc}, or to the method suggested in~\cite{ostertag2022trajectory}.

\bibliographystyle{ieeetr}
\bibliography{\BIBTEXDIR/biblio}

\begin{thebibliography}{10}

\bibitem{smith2011persistentOcean}
R.~N. Smith, M.~Schwager, S.~L. Smith, B.~H. Jones, D.~Rus, and G.~S. Sukhatme,
  ``Persistent ocean monitoring with underwater gliders: Adapting sampling
  resolution,'' {\em Journal of Field Robotics}, vol.~28, no.~5, pp.~714--741,
  2011.

\bibitem{alam2018data}
T.~Alam, G.~M. Reis, L.~Bobadilla, and R.~N. Smith, ``A data-driven deployment
  approach for persistent monitoring in aquatic environments,'' in {\em 2018
  second IEEE international conference on robotic computing (IRC)},
  pp.~147--154, IEEE, 2018.

\bibitem{naderi2022sharing}
S.~Naderi, K.~Bundy, T.~Whitney, A.~Abedi, A.~Weiskittel, and A.~Contosta,
  ``Sharing wireless spectrum in the forest ecosystems using artificial
  intelligence and machine learning,'' {\em International Journal of Wireless
  Information Networks}, vol.~29, no.~3, pp.~257--268, 2022.

\bibitem{casbeer2006cooperative}
D.~W. Casbeer, D.~B. Kingston, R.~W. Beard, and T.~W. McLain, ``Cooperative
  forest fire surveillance using a team of small unmanned air vehicles,'' {\em
  International Journal of Systems Science}, vol.~37, no.~6, pp.~351--360,
  2006.

\bibitem{zhu2022control}
Y.~Zhu and S.~B. Andersson, ``Control policy optimization for data harvesting
  in a wireless sensor network,'' in {\em 2022 IEEE 61st Conference on Decision
  and Control (CDC)}, pp.~7437--7442, IEEE, 2022.

\bibitem{lee2006efficient}
U.~Lee, E.~Magistretti, B.~Zhou, M.~Gerla, P.~Bellavista, and A.~Corradi,
  ``Efficient data harvesting in mobile sensor platforms,'' in {\em Fourth
  Annual IEEE International Conference on Pervasive Computing and
  Communications Workshops (PERCOMW'06)}, pp.~5--pp, IEEE, 2006.

\bibitem{pinto2021tracking}
S.~C. Pinto, N.~A. Vickers, F.~Sharifi, and S.~B. Andersson, ``Tracking
  multiple diffusing particles using information optimal control,'' in {\em
  2021 American Control Conference (ACC)}, pp.~4033--4038, IEEE, 2021.

\bibitem{cortes2004coverage}
J.~Cortes, S.~Martinez, T.~Karatas, and F.~Bullo, ``Coverage control for mobile
  sensing networks,'' {\em IEEE Transactions on robotics and Automation},
  vol.~20, no.~2, pp.~243--255, 2004.

\bibitem{cassandras2011optimal}
C.~G. Cassandras, X.~C. Ding, and X.~Lin, ``An optimal control approach for the
  persistent monitoring problem,'' in {\em 2011 50th IEEE conference on
  decision and control and european control conference}, pp.~2907--2912, IEEE,
  2011.

\bibitem{boldrer2022time}
M.~Boldrer, L.~Lyons, L.~Palopoli, D.~Fontanelli, and L.~Ferranti,
  ``Time-inverted kuramoto model meets lissajous curves: Multi-robot persistent
  monitoring and target detection,'' {\em IEEE Robotics and Automation
  Letters}, vol.~8, no.~1, pp.~240--247, 2022.

\bibitem{lin2014optimal}
X.~Lin and C.~G. Cassandras, ``An optimal control approach to the multi-agent
  persistent monitoring problem in two-dimensional spaces,'' {\em IEEE
  Transactions on Automatic Control}, vol.~60, no.~6, pp.~1659--1664, 2014.

\bibitem{yu2015persistent}
J.~Yu, S.~Karaman, and D.~Rus, ``Persistent monitoring of events with
  stochastic arrivals at multiple stations,'' {\em IEEE Transactions on
  Robotics}, vol.~31, no.~3, pp.~521--535, 2015.

\bibitem{hari2020optimal}
S.~K.~K. Hari, S.~Rathinam, S.~Darbha, K.~Kalyanam, S.~G. Manyam, and
  D.~Casbeer, ``Optimal {UAV} route planning for persistent monitoring
  missions,'' {\em IEEE Transactions on Robotics}, vol.~37, no.~2,
  pp.~550--566, 2020.

\bibitem{anderson2014stochastic}
R.~P. Anderson and D.~Milutinovi{\'c}, ``A stochastic approach to dubins
  vehicle tracking problems,'' {\em IEEE Transactions on Automatic Control},
  vol.~59, no.~10, pp.~2801--2806, 2014.

\bibitem{hall2022optimal}
J.~Hall, S.~B. Andersson, and C.~G. Cassandras, ``Optimal persistent monitoring
  of mobile targets in one dimension,'' {\em arXiv preprint arXiv:2210.01294},
  2022.

\bibitem{wang2023spatio}
Y.~Wang, Y.~Wang, Y.~Cao, and G.~Sartoretti, ``Spatio-temporal attention
  network for persistent monitoring of multiple mobile targets,'' {\em arXiv
  preprint arXiv:2303.06350}, 2023.

\bibitem{kant1986toward}
K.~Kant and S.~W. Zucker, ``Toward efficient trajectory planning: The
  path-velocity decomposition,'' {\em The international journal of robotics
  research}, vol.~5, no.~3, pp.~72--89, 1986.

\bibitem{smith2011persistent}
S.~L. Smith, M.~Schwager, and D.~Rus, ``Persistent monitoring of changing
  environments using a robot with limited range sensing,'' in {\em 2011 IEEE
  International Conference on Robotics and Automation}, pp.~5448--5455, IEEE,
  2011.

\bibitem{song2014optimal}
C.~Song, L.~Liu, G.~Feng, and S.~Xu, ``Optimal control for multi-agent
  persistent monitoring,'' {\em Automatica}, vol.~50, no.~6, pp.~1663--1668,
  2014.

\bibitem{ostertag2022trajectory}
M.~Ostertag, N.~Atanasov, and T.~Rosing, ``Trajectory planning and optimization
  for minimizing uncertainty in persistent monitoring applications,'' {\em
  Journal of Intelligent \& Robotic Systems}, vol.~106, no.~1, pp.~1--19, 2022.

\bibitem{notomista2021online}
G.~Notomista, C.~Pacchierotti, and P.~R. Giordano, ``Online robot trajectory
  optimization for persistent environmental monitoring,'' {\em IEEE Control
  Systems Letters}, vol.~6, pp.~1472--1477, 2021.

\bibitem{alamdari2014persistent}
S.~Alamdari, E.~Fata, and S.~L. Smith, ``Persistent monitoring in discrete
  environments: Minimizing the maximum weighted latency between observations,''
  {\em The International Journal of Robotics Research}, vol.~33, no.~1,
  pp.~138--154, 2014.

\bibitem{welikala2021event}
S.~Welikala and C.~G. Cassandras, ``Event-driven receding horizon control for
  distributed persistent monitoring in network systems,'' {\em Automatica},
  vol.~127, p.~109519, 2021.

\bibitem{rao2009survey}
A.~V. Rao, ``A survey of numerical methods for optimal control,'' {\em Advances
  in the Astronautical Sciences}, vol.~135, no.~1, pp.~497--528, 2009.

\bibitem{nurkanovic2022nosnoc}
A.~Nurkanovi{\'c} and M.~Diehl, ``{NOSNOC}: A software package for numerical
  optimal control of nonsmooth systems,'' {\em IEEE Control Systems Letters},
  2022.

\bibitem{hall2021sequential}
J.~Hall, A.~Nurkanovi{\'c}, F.~Messerer, and M.~Diehl, ``A sequential convex
  programming approach to solving quadratic programs and optimal control
  problems with linear complementarity constraints,'' {\em IEEE Control Systems
  Letters}, vol.~6, pp.~536--541, 2021.

\bibitem{cassandras2012optimal}
C.~G. Cassandras, X.~Lin, and X.~Ding, ``An optimal control approach to the
  multi-agent persistent monitoring problem,'' {\em IEEE Transactions on
  Automatic Control}, vol.~58, no.~4, pp.~947--961, 2012.

\bibitem{zhou2018optimal}
N.~Zhou, X.~Yu, S.~B. Andersson, and C.~G. Cassandras, ``Optimal event-driven
  multiagent persistent monitoring of a finite set of data sources,'' {\em IEEE
  Transactions on Automatic Control}, vol.~63, no.~12, pp.~4204--4217, 2018.

\bibitem{pinto2022multi}
S.~C. Pinto, S.~B. Andersson, J.~M. Hendrickx, and C.~G. Cassandras,
  ``Multi-agent persistent monitoring of targets with uncertain states,'' {\em
  IEEE Transactions on Automatic Control}, 2022.

\bibitem{beaver2023IFAC}
L.~E. Beaver, R.~Tron, and C.~G. Cassandras, ``A graph-based approach to
  generate energy-optimal robot trajectories in polygonal environments,'' {\em
  IFAC-PapersOnLine (to appear)}, 2023.
\newblock 22nd IFAC World Congress (to appear).

\bibitem{nocedal1999numerical}
J.~Nocedal and S.~J. Wright, {\em Numerical optimization}.
\newblock Springer, 1999.

\bibitem{andersson2019casadi}
J.~A. Andersson, J.~Gillis, G.~Horn, J.~B. Rawlings, and M.~Diehl, ``Casadi: a
  software framework for nonlinear optimization and optimal control,'' {\em
  Mathematical Programming Computation}, vol.~11, no.~1, pp.~1--36, 2019.

\bibitem{wachter2006implementation}
A.~W{\"a}chter and L.~T. Biegler, ``On the implementation of an interior-point
  filter line-search algorithm for large-scale nonlinear programming,'' {\em
  Mathematical programming}, vol.~106, pp.~25--57, 2006.

\end{thebibliography}

\newpage

\end{document}